\documentclass[a4paper,10pt]{amsart}

\usepackage[latin1]{inputenc}
\usepackage[T1]{fontenc}
\usepackage[english]{babel}

\usepackage{mathrsfs}
\usepackage{amscd}
\usepackage{amsfonts}
\usepackage{amsmath}
\usepackage{amssymb}
\usepackage{amstext}
\usepackage{amsthm}
\usepackage{amsbsy}

\usepackage{xspace}
\usepackage[all]{xy}
\usepackage{graphicx}
\usepackage{url}
\usepackage{latexsym}

\newtheorem{theo}{Theorem}[section]
\newtheorem{lem}[theo]{Lemma}
\newtheorem{prop}[theo]{Proposition}

\newtheorem{fact}[theo]{Fact}
\newtheorem{conj}[theo]{Conjecture}
\newtheorem{que}[theo]{Question}

\theoremstyle{definition}
\newtheorem{dfn}[theo]{Definition}
\newtheorem{rem}[theo]{Remark}
\newtheorem{ex}[theo]{Example}

\newcommand{\N}{\ensuremath{\mathbb{N}}}  
\newcommand{\Z}{\ensuremath{\mathbb{Z}}}

\newcommand{\R}{\ensuremath{\mathbb{R}}}

\newcommand{\M}{\ensuremath{\mathcal{M}}}

\newcommand{\vs}{\vspace{0.2cm}}

\DeclareMathOperator{\Aut}{Aut}

\DeclareMathOperator{\SO}{SO}

\DeclareMathOperator{\id}{id}

\begin{document}

\author{Annalisa Conversano} 
 
\title{Nilpotent groups, o-minimal Euler characteristic, and linear algebraic groups}  

\address{Massey University Auckland, New Zealand} 

\email{a.conversano@massey.ac.nz}

\date{September 2, 2020. Research supported in part by the the Hausdorff Institute in Bonn.}
\maketitle

\vspace{-.5cm}
\begin{abstract} 
We establish a surprising correspondence between groups definable in o-minimal structures 
and linear algebraic groups, in the nilpotent case. 
It turns out that in the o-minimal context, like for finite groups, nilpotency is equivalent to the normalizer property or to uniqueness of Sylow subgroups.
As a consequence, we show definable algebraic decompositions of o-minimal nilpotent groups, and we prove that a nilpotent Lie group is definable in an o-minimal expansion of the reals if and only if it is a linear algebraic group.  
\end{abstract}

\thispagestyle{empty}

\vs
\section{Introduction}

Groups that are definable in o-minimal structures have been studied by many authors in the past thirty years, often in analogy with Lie groups.

By a conjecture of Pillay in \cite{PC}, now fully proved, every definable group $G$ has a canonical quotient $G/G^{00}$ that, endowed with the logic topology, is a compact Lie group \cite{BOPP}. When $G$ is definably compact, $G$ and $G/G^{00}$ have same dimension \cite{NIPI}, same homotopy invariants \cite{Berarducci-Mamino, Berarducci-Mamino-Otero}, and same first order theory \cite{HPP}. 

Strong connections with Lie groups have been found also for groups that are not compact.
For instance, every connected abelian real Lie group is the direct product of its maximal torus $T$ by a torsion-free closed subgroup. Similarly, 
by \cite{me2}, every o-minimal definably connected abelian group $G$ is the direct product of a maximal abstract torus $T$ (Definition \ref{abstract-dfn}) and the maximal torsion-free definable subgroup $\mathcal{N}(G)$ (Fact \ref{me-solvable}). Therefore every abelian o-minimal group is elementarily equivalent to a linear algebraic group of the same dimension. This is not the case, in general, for solvable groups, as shown by Hrushovski, Peterzil and Pillay in \cite{HPP}. They give an example of a solvable o-minimal group that is not elementarily equivalent to any definable real Lie group.  In this paper we study the intermediate class of nilpotent groups, showing a surprising similarity with the linear algebraic setting, even for finite groups. In Section 2 we prove the following: 


 \begin{theo}\label{first}
Let $G$ be a nilpotent definable group. Then


\begin{enumerate}

%
\item $G$ has maximal  abstractly compact subgroups $K$ such that 
\[
G = K \times \mathcal{N}(G)
\]

\noindent where $\mathcal{N}(G)$ is the maximal normal definable torsion-free subgroup of $G$.\\

\item If $G$ is definably connected then its center $Z(G)$ is definably connected and contains every abstractly compact subgroup of $G$.

\end{enumerate}

\end{theo}

\noindent
As a consequence of decomposition $(1)$ above, in Section 4 we show that
linear algebraic groups are the only nilpotent Lie groups that can be defined in an o-minimal expansion of the real field:


 \begin{theo} \label{Lie}
Let $G$ be a nilpotent real Lie group. Then $G$ is definable in an o-minimal structure over the reals
if and only if $G$ is Lie isomorphic to a linear algebraic group.  \vs
\end{theo}

A main tool is the o-minimal Euler characteristic $E$, an invariant under definable bijections  that has been used by Strzebonski in \cite{Strzebonski} to develop a theory of definable $p$-groups and definable $p$-Sylow subgroups, extending classical notions and results for finite groups. 
In Section 2 and 3 it is used to show the following equivalent characterizations to nilpotency, well-known for finite groups:
 
 \begin{theo}\label{finite-like}
Let $G$ be a definable group such that $\mathcal{N}(G)$ is nilpotent. 

\begin{enumerate}
\item Assume $E(G) \neq 0$. Then the following are equivalent:

 \begin{enumerate}
\item[(a)] $G$ is nilpotent.

 \item[(b)] $G$ has exactly one $p$-Sylow subgroup for each prime $p$ dividing $E(G)$.

\item[(c)] All $p$-Sylow subgroups of $G$ are normal.


\end{enumerate}

\vs
\item Suppose $E(G) = 0$ and $G = G^0$. Then the following are equivalent:

\begin{enumerate}
\item[(a)] $G$ is nilpotent.

\item[(b)] $G$ has exactly one $0$-Sylow subgroup.

\item[(c)] All $0$-Sylow subgroups of $G$ are normal.


\end{enumerate}

\vs
\item Let $G$ be definably connected. Then the following are equivalent:

\begin{enumerate}
\item[(a)] $G$ is nilpotent.

\item[(b)] Every proper definable $H < G$ is contained properly in its normalizer. 


\end{enumerate}

\end{enumerate}
\end{theo}
 
 \noindent
 Note that in Theorem \ref {finite-like} the assumption $\mathcal{N}(G)$ nilpotent is necessary in (1) and (2), as there are torsion-free definable groups that are not nilpotent like, for instance, the centerless semialgebraic group
 \[
 H = \left \{ \begin{pmatrix} 
 a & b \\
 0 & 1
 \end{pmatrix}: a, b \in \R, a > 0 \right \} 
 \]
 
\noindent For any finite nilpotent group $F$, the semialgebraic group $G = H \times F$ 
 satisfies $(1)$ b\&c and the semialgebraic group $G = H \times \SO_2(\R)$ satisfies $(2)$ b\&c. In both cases, $\mathcal{N}(G) = H$. We do not know whether the assumption can be eliminated in $(3)$ (Note that $H$ above does not give us this information for  (3) as the definable subgroup of matrices where $b = 0$ is equal to its normalizer).\\

\noindent
Finally, Section 4 contains a digression on definable abelian torsion-free groups $G$, for which a decomposition  in 1-dimensional definable subgroups  is proved, when $\dim \Aut(G)> 0$. This is related to the problem of characterizing definable groups that are elementarily equivalent to a linear algebraic group of the same dimension.  

Throughout the paper groups are definable \emph{with parameters} in an o-minimal structure \M. We assume \M\ expands a real closed field, since some of the references in the proofs are stated under this assumption, but we believe all results are likely to hold in an arbitrary o-minimal structure.


 \section{Nilpotency and Euler characteristic}

If $\mathcal{P}$ is a cell decomposition of a definable set $X$, \emph{the o-minimal Euler characteristic} $E(X)$ is defined as the number of even-dimensional cells in 
$\mathcal{P}$ minus the number of odd-dimensional cells in $\mathcal{P}$, and it does not depend on $\mathcal{P}$ (see \cite{Lou}, Chapter 4). As points are $0$-dimensional cells, it follows that for finite sets cardinality and Euler characteristic coincide. Moreover, since for every definable sets $A$, $B$ we have that $E(A \times B) = E(A)E(B)$, the following holds:

\begin{fact}\cite{Strzebonski} \label{products} 
Let $K < H < G$ be definable groups. Then
\begin{enumerate}
\item[(a)] $E(G) = E(H)E(G/H)$

\item[(b)] $E(G/K) = E(G/H)E(H/K)$
\end{enumerate}
\end{fact}

\begin{dfn}\cite{Strzebonski} Let $G$ be a definable group. We say that $G$ is a \emph{$p$-group} if:
\begin{itemize}
\item $p$ is a prime number and for any proper definable $H < G$,
\[
E(G/H) \equiv 0 \quad \mod p
\]

\item $p = 0$ and for any proper definable subgroup $H < G$,
\[
E(G/H) = 0
\]

\end{itemize}

A maximal $p$-subgroup of a definable group $G$ is called \emph{$p$-Sylow}.
\end{dfn}

\begin{fact}\label{fact} \cite{Strzebonski} Let $G$ be a definable group.
\begin{enumerate}
\item If $p$ is a prime dividing $E(G)$, then $G$ contains an element of order $p$. In particular, if $E(G) = 0$ then $G$ has elements of each prime order. Moreover, 
\[
 G \mbox{ is torsion-free } \ \Longleftrightarrow\ E(G) = \pm 1
\]

\item Each $p$-subgroup is contained in a $p$-Sylow, and $p$-Sylows are all conjugate.   

\item If $H$ is a $p$-subgroup of $G$, then 
\[
H \mbox{ is a $p$-Sylow } \Longleftrightarrow\ E(G/H) \neq 0 \mod p
\]


\item If $E(G) = 0$, then $G$ contains a $0$-subgroup.  

\item Every $0$-group is abelian and definably connected.  

\item If $E(G) \neq 0$, then any $p$-subgroup of $G$ is finite.  

\item Let $S \subset G$ be a subset (definable or not). Then there is a smallest definable subgroup $H < G$ containing $S$. We call it the definable subgroup generated by $S$, and we write $H = \langle S \rangle$.
\end{enumerate}
\end{fact}

Given a definable group $G$, we denote by $\mathcal{N}(G)$ the maximal normal definable torsion-free subgroup of $G$ (that exists by Proposition 2.1 in \cite{Conversano-PillayI}). The smallest definable subgroup of finite index in $G$ is denoted by $G^0$ and it exists because of DCC for definable subgroups \cite[Theo 2.6]{Strzebonski}. Moreover, the following hold:  

\begin{fact} \cite {Pillay - groups} \quad
$ 
G \mbox{ is definably connected }  \Longleftrightarrow\ G = G^0.
$
\end{fact}

\noindent
That is, the definably connected component of the identity in $G$ is a definable subgroup.  As observed in the introduction of \cite{PPSI}:

\begin{fact} \label{tricotomy}
If $G$ is definably connected, then either $E(G) = \pm 1$ (iff $G$ is torsion-free) or $E(G) = 0$.
\end{fact}

\begin{fact} \label{tfref}\cite{PeSta} Let $G$ be a definable torsion-free group. 

\begin{enumerate}
\item $G$ is definably connected and solvable.

\item $G$ is definably contractible. Namely, there is a definable homotopy \\
$\mathcal{H} \colon G \times [0, 1] \to G$ between the identity map on $G$ and the function $G \to G$ taking the constant value $e \in G$.   
\end{enumerate}
\end{fact}


We first consider the case where $E(G)\neq 0$.
Because of Fact \ref{tricotomy} and Fact \ref{fact}(1), if $E(G) \notin \{-1, 0, 1\}$ then $G$ is not definably connected and $G^0 = \mathcal{N}(G)$ is torsion-free.
 
%
 
 \begin{lem} \label{p-group}
 Let $G$ be a definable group such that $|E(G)| = p^a$, for some $p$ prime. Then any $p$-Sylow subgroup $H$ of $G$ has order $p^a$, $H$ is definably isomorphic to $G/G^0$ and $G = G^0 \rtimes H$. 
 \end{lem}
 
 \begin{proof}
 As $E(G) = E(G^0) E(G/G^0)$ and $E(G^0) = \pm 1$, it follows that $|E(G)| = E(G/G^0) = |G/G^0| = p^a$, as $G/G^0$ is a finite group. 
 
Let $H$ be a $p$-Sylow subgroup of $G$. By Fact \ref{fact}(3) we know that $E(G/H) \neq 0  \mod p$, thus $E(G/H) = \pm 1$. So $E(H) = |H| = |E(G)| = p^a$. Moreover, $G^0$ and $H$ have trivial intersection, as $G^0$ is torsion-free and $H$ is finite. Therefore $G = G^0 \rtimes H$, as wanted.
\end{proof}
 
 \begin{rem} \label{semidirect}

The semidirect product may be not direct. E.g.,  $G = \R \rtimes \Z_2$ (where $\Z_2 = \{\pm 1\}$ acts on \R\ by multiplication) is a centerless group with $E(G) = -2$.
 
 \end{rem}
 
 But when $G$ is nilpotent, much more can be said:
 
\begin{prop} \label{str-nilpotent}
 Let $G$ be a nilpotent definable group such that $E(G) \neq 0$. Then
 \begin{enumerate}
 \item the center $Z(G)$ is infinite whenever $G$ is infinite;
 
 \item for each $p$ prime dividing $|E(G)|$, $G$ has exactly one $p$-Sylow subgroup;
 
\item $G = F \times \mathcal{N}(G)$, where $F$ is the direct product of the (unique) $p$-Sylow subgroups of $G$. 
 \end{enumerate}
 \end{prop}
 
 \begin{proof}
If $G$ is finite, then $\mathcal{N}(G) = \{e\}$, and $(2)$ and $(3)$ are well-known. So let $G$ be infinite with $\dim G = n > 0$ and $|E(G)| = m = p_1^{a_1} \cdots p_k^{a_k}$. We will prove the three statements by induction on $n + m$. 

Suppose, for a contradiction, that $Z = Z(G)$ is finite of cardinality $r$. Then $G/Z$ is a nilpotent group of dimension $n$ and Euler characteristic 
$m/r < m$. By induction, $G/Z = F' \times N'$, where $N' = \mathcal{N}(G/Z)$ and $F'$ is the direct product 
of its unique $p$-Sylow subgroups. Let now $F$ be the pull-back in $G$ of $F'$. This is a finite nilpotent group so it is the direct product of its unique $p$-Sylow subgroups and $G = \mathcal{N}(G) \times F$. However this implies that the infinite center of $\mathcal{N}(G)$ is included in the center of $G$ that was assumed to be finite, contradiction. So $Z(G)$ is infinite and $(1)$ holds.  
 
Now assume $Z(G)^0 = G^0$. If $k = 1$ and $|E(G)| = p^a$, then by Lemma \ref{p-group} we know that $G = G^0 \rtimes G/G^0$. But as $Z(G)^0 = G^0$, the product is direct and $G$ has exactly one $p$-Sylow subgroup. 

Suppose $k > 1$. As $G/G^0$ is a finite nilpotent group, it is the direct product of its (unique) $p_i$-Sylow subgroups $H_1, \dots, H_k$. Let $K_1 < G$ be the pull-back of the product of the first $k-1$ factors, and $K_2$  be the pull-back of $H_k$.
By induction $K_1 = G^0 \times F_1 \times \cdots \times F_{k-1}$ and $K_2 = G^0 \times F_k$, where each $F_i$ is the unique $p_i$-Sylow in $G$ (therefore normal) and $F$ in $(3)$ is the product $F_1 \times \cdots \times F_k$.

Finally, assume $Z = Z(G)^0 \subsetneq G^0$ and let $G_1 = G/Z$. As $E(G) \neq 0$, then $G^0 = \mathcal{N}(G)$ is torsion-free and $E(Z) = \pm 1$. 
Then $|E(G_1)| = |E(G)| = m$ and $\dim G_1 < \dim G$.
By induction $G_1 = F' \times G_1^0$, where $F'$ is the direct product of its (unique) $p_i$-Sylow subgroups ($i = 1, \dots, k$). Let now $K$ be the pull-back in $G$ of $F'$.
Then, by the previous case, $K = Z(G)^0 \times F$. As $G/K = G_1^0$ is torsion-free, all $p$-subgroups of $G$ are contained in $K$, so $(2)$ and $(3)$ hold.  
 \end{proof}

\begin{rem} In the proposition above,
$G$ nilpotent is an essential assumption for all three conditions. For conditions (2) and (3), we have already noticed this in Remark \ref{semidirect}. For condition (1), it is enough to consider a definable centerless torsion-free group, such as  $\R \rtimes \R^{>0}$.
\end{rem}
 

We can now show the first part of Theorem \ref{finite-like}:

\begin{proof} [Proof of Theorem  \ref{finite-like}(1)] Suppose $G$ is a definable group such that $E(G) \neq 0$.
\begin{enumerate}
\item[$(a) \Rightarrow (b)$] If $G$ is nilpotent, then by Proposition \ref{str-nilpotent}, $G$ has exactly one $p$-Sylow subgroup for each $p$ prime dividing $E(G)$.

\item [$(b) \Rightarrow (c)$] Obvious. 

\item [$(c) \Rightarrow (a)$] Suppose all $p$-Sylow subgroups of $G$ are normal, and let $H$ be their product. Clearly $H$ is a normal subgroup of $G$ and $\mathcal{N}(G) \cap H = \{e\}$, since all $p$-subgroups of $G$ are finite by Fact 2.3(4). Therefore $G = H \times \mathcal{N}(G)$.
As finite $p$-groups are nilpotent and we are assuming $\mathcal{N}(G)$ is nilpotent, it follows that $G$ is nilpotent as well.
\end{enumerate} \vspace{-.5cm}
\end{proof}

We now consider the case where $E(G) = 0$. It is well-known that $G$ may have no maximal \emph{definably compact} subgroup (for instance, see Example
\ref{str-ex} from \cite{Strzebonski}). However, by Theorem 1.5 in \cite{me2}, if $G$ is definably connected then $G$ always has  maximal \emph{abstractly compact} subgroups, all conjugate (if definable). We will show that when $G$ is nilpotent (definably connected or not), then maximal abstractly compact subgroups of $G$ are a direct complement of $\mathcal{N}(G)$. 
 


\begin{dfn} \label{abstract-dfn}
Let $G$ be a definable group. We say that a subgroup $K < G$ is \emph{abstractly compact} if there is a definable homomorphism $G \to G_1$ ($G_1$ definable group)  whose restriction to $K$ is an isomorphism with a definably compact definable subgroup of $G_1$.
In other words, there is a definable normal subgroup $N$ of $G$ and a definably compact subgroup $K'$ of $G/N$, whose pull-back in $G$ is $N \rtimes K$. 
%

We call $K$ an \emph{abstract torus} when $K'$ is a definable torus (that is, abelian, definably connected and definably compact).
%
\end{dfn}

 \begin{ex} \label{str-ex} (\cite[5.3]{Strzebonski}) Let $\M$ be the real field and $G = \R \times [1, e[$ with the operation defined by
 \[
 (x, u) \ast (y, v) = 
 \begin{cases}
 (x + y, uv)  & \mbox{ if $uv < e$} \\
 (x+y+1, uv/e) & \mbox{ otherwise}
 \end{cases}
 \]
 
 \vspace{0.2cm} \noindent
 $(G, \ast)$ is a $0$-group with $\mathcal{N}(G) = \R \times \{1\}$. The only definably compact subgroups of $G$ are finite. The subgroup $T = \{(-\ln u, u): u \in [1, e[\}$ is an abstract torus isomorphic to the definable torus $([1, e[, \otimes)$, where $\otimes$ denotes the multiplication$\mod e$, via the canonical projection $\pi \colon G \to G/\mathcal{N}(G)$.
 \end{ex}

\begin{fact}\label{me-solvable} \cite{me2} Let $G$ be a solvable definably connected definable group. Then
\begin{itemize}
\item $G/\mathcal{N}(G)$ is a definable torus (and the maximal definable quotient of $G$ that is definably compact).

\item For any $0$-Sylow $A$ of $G$, $G = \mathcal{N}(G) \cdot A = \mathcal{N}(G) \rtimes T$, where $T$ is any direct complement of $\mathcal{N}(A)$ in $A$. 
\end{itemize}

In particular, if $G$ is abelian, then $A$ is unique (Fact \ref{fact}(2)) and $G = \mathcal{N}(G) \times T$. 
\end{fact}
The following lemma was suggested by interesting questions from the referee:

\begin{lem} \label{torsion}
 Let $G$ be a definably connected abelian definable group. Let $S$ be the torsion subgroup  and $A$ the $0$-Sylow of $G$. Then
 
 \begin{enumerate}
 \item[(i)] $G$ has a unique maximal abstract torus $T$.  
 
 \item[(ii)] $\langle S \rangle = \langle T \rangle = A$. 
 \end{enumerate}
\end{lem}

\begin{proof}
 \begin{enumerate}
 \item[(i)] As reminded in Fact \ref{me-solvable}, a maximal abstract torus $T$ of $G$ is found by taking a direct complement of $\mathcal{N}(A) = \mathcal{N}(G) \cap A$ in $A$, that we know exists because both $\mathcal{N}(A) $ and $A$ are divisible abelian groups. 
 Let $T_1$ be another maximal abstract torus of $G$ and $f_1 \colon G \to G_1$ be a definable homomorphism that restricted to $T_1$ is an isomorphism with a definable torus $\bar{T}_1 \subset G_1$. Since $\bar{G} = G/\mathcal{N}(G)$ is the maximal quotient of $G$ that is definably compact (that is, it contains every definably compact quotient of $G$), we can assume $G_1 = \bar{G}$ and $f_1 =   \pi \colon G \to \bar{G}$, to be the canonical projection.
 
We claim that $\pi(T_1) = \bar{G}$. If not, by divisibility again, $\pi(T) = \bar{G} = \pi(T_1) \times \bar{K}$, for some infinite subgroup $\bar{K}$. Let now $K$ be the image in $T$ of $\bar{K}$ by an isomorphism between $T$ and $\bar{G}$. Then $T_1 \times K$ is an abstract torus containing properly $T_1$, contradiction.

As the restriction of $\pi$ to $T$ is injective, it follows that $\pi(T \cap T_1) = \pi (T) \cap \pi(T_1) = \bar{G}$. Therefore $T \cap T_1$ is a maximal abstract torus, and $T = T_1$. \\


 
 \item[(ii)] Set $K = \langle S \rangle$ and $H =  \langle T \rangle$. As $S \subset T \subset A$, we know that $K \subseteq H \subseteq A$.  
 


Since $K$ is definably connected, it is divisible and  $A = K \times K_1$, for some $K_1 < A$. As $K$ contains all torsion elements of $A$, it follows that $K_1$ is torsion-free and therefore $E(A/K) = \pm 1$. However, $A$ is a $0$-group, so $K_1 = \{e\}$ and $K = A$. 
\end{enumerate}
\end{proof}

The following proposition shows that nilpotent definably connected groups are essentially abelian modulo the maximal torsion-free definable subgroup:

\begin{prop}\label{decomposition}
Let $G$ be a  nilpotent group such that $E(G) = 0$ and let $S$ be the torsion subgroup of $G$. Suppose $G$ is definably connected. Then

\begin{enumerate}
\item $G$ has a unique $0$-Sylow subgroup $A$ and it is contained in the center of $G$;

\item $\langle S \rangle = A$;  

\item $G$ has a unique maximal abstract torus $T \subseteq A$ and $T$ is a direct complement of $\mathcal{N}(G)$. Namely, 
\[G = \mathcal{N}(G) \times T \]
\end{enumerate}

\end{prop}

\begin{proof} 
By induction on $n = \dim G$. If $n = 1$, then by Fact \ref{fact}(4)(5), $G$ is a $0$-group and there is nothing to prove. Suppose $n > 1$. If $G$ is abelian, see Fact \ref{me-solvable}. So let $G$ be non-abelian and set $Z = Z(G)$. Note that both $Z$ and $G/Z$ are infinite, because $G$ is nilpotent and definably connected. 

If $G/Z$ is definably compact, then $\mathcal{N}(G/Z)$ is trivial and  $\mathcal{N}(G) \subseteq Z$. Therefore $G = \mathcal{N}(G) \times T$ 
(Fact \ref{me-solvable}) and $G$ has a unique $0$-Sylow $A$ that is, moreover, central. 

So let $G/Z$ be not definably compact.  We distinguish the two cases where $E(G/Z) = \pm 1$ or $E(G/Z) = 0$. In the first case, $G/Z$ is torsion-free and $S \subset Z$. Note that the unique $0$-Sylow of $Z$ is the $0$-Sylow of $G$ as well by Fact \ref{fact}(3),
and the other claims follow by Lemma \ref{torsion}.

If $E(G/Z) = 0$, by induction $G/Z$ has a unique $0$-Sylow $A_1$, and $G/Z = N_1 \times T_1$, 
 where $N_1 = \mathcal{N}(G/Z)$ is definable torsion-free, and $T_1 \cong A_1/\mathcal{N}(A_1)$ is a maximal abstract torus of $G/Z$. Note that $A_1$ is the image of any $0$-Sylow $A$ of $G$. By Proposition 2.6 in \cite{me2}, $\mathcal{N}(A)$ is central in $G$, therefore $A_1$ is definably compact and $T_1 = A_1$.

Let $H$ be the pull-back of $A_1$ in $G$. As $A_1$ is normal, $H$ is normal as well. The quotient $G/H = N_1$ is torsion-free, so $H$ contains all $0$-subgroups of $G$. By induction (since $\dim N_1 > 0$, as $G/Z$ is not definably compact), $H$ has a unique $0$-Sylow, so $G$ has a unique $0$-Sylow $A = \mathcal{N}(A) \times T$.

Since $A$ is the unique $0$-Sylow, then $A$ is normal in $G$. Note that the definable group $G/A$ is definably isomorphic to $\mathcal{N}(G)/\mathcal{N}(A)$(\ref{me-solvable}), so torsion-free. It follows that $A$ contains all $k$-torsion elements $G[k]$ of $G$, for each $k \in \N$, and $A[k] = G[k]$.
Each $A[k]$ is a finite normal subgroup of $G$, therefore central. Therefore $S$ is a central subgroup of $G$ (and coincides with the torsion subgroup of $A$). 
By Lemma \ref{torsion}, $\langle S \rangle = A$ and $A$ is central as well.
The unique maximal abstract torus $T$ of $A$ from Lemma \ref{torsion} is the unique maximal abstract torus of $G$. As $T$ is isomorphic to $G/\mathcal{N}(G)$ (Fact \ref{me-solvable}), then $G = \mathcal{N}(G) \times T$.
 \end{proof}


\begin{rem}
The nilpotency assumption in Proposition \ref{decomposition} cannot be extended to solvability, not even for linear groups. For instance, the group $G = \R^2 \rtimes \SO_2(\R)$, where $\SO_2(\R)$ acts on $\R^2$ by matrix multiplication, is a centerless solvable linear group with several $0$-Sylows.  
 \end{rem}

We now show the second part of Theorem \ref{finite-like}:
 
\begin{proof}[Proof of Theorem \ref{finite-like}(2)] Let $G$ be a definably connected group with $E(G) = 0$.
\begin{enumerate}
\item[$(a) \Rightarrow (b)$] If $G$ is nilpotent, then by Proposition \ref{decomposition}, $G$ has exactly one $0$-Sylow.  

\item [$(b) \Rightarrow (c)$] Obvious. 

\item [$(c) \Rightarrow (a)$]  By Theorem 1.5 in \cite{me2}, $G = PH$ where $P$ is a union of conjugates of a $0$-Sylow $A$ and $H$ is definable torsion-free.   Since $A$ is normal in $G$ by assumption, then $P = A$ and $G$ is solvable. Whenever $G$ is solvable and definably connected, then $G/\mathcal{N}(G)$ is definably compact and therefore abelian by  \cite{Peterzil-Starchenko1}. As we are assuming $\mathcal{N}(G)$ nilpotent, then $G$ is nilpotent as well.
\end{enumerate} \vspace{-.2cm}
 \end{proof}

We conclude the section with the proof of Theorem \ref{first}:

\begin{proof}[Proof of Theorem \ref{first}] Let $G$ be a nilpotent definable group.
\begin{enumerate}
 
\item We want to show that $G$ has maximal abstractly compact subgroups $K$, and any such $K$ is a direct complement of $\mathcal{N}(G)$. If $E(G) \neq 0$, then $K = F$ from Proposition \ref{str-nilpotent}. If $G = G^0$ and $E(G) = \pm 1$, then $K = \{e\}$.
If $G = G^0$ and $E(G) = 0$, then $K = T$ in Proposition \ref{decomposition}. If $E(G) = 0$ and $G \neq G^0$, then $K = F \cdot T$, where $F$ is a finite normal subgroup of $G$ such that $G = F \cdot G^0$ \cite[Theo 6.10]{Edmundo}, and $T$ is the  maximal abstract torus of $G^0$ from Proposition \ref{decomposition}.

\item If $G$ is definably connected, then $E(G) = \pm 1$ or $E(G) = 0$. 
Set $N = \mathcal{N}(G)$. If $E(G) = \pm 1$, then $G = N$ and $\{e\}$ is the only abstractly compact subgroup of $G$. If $E(G) = 0$,
by Proposition \ref{decomposition}, 
\[
Z(G) = Z(N) \times T
\]
\noindent where $T$ is the maximal abstract torus of $G$. Therefore $Z(G)$ is definably connected and contains every abstractly compact subgroup of $G$.  
\end{enumerate}\vspace{-.5cm}
\end{proof}

\section{Nilpotency and normalizers}

It is well-known that a finite group $G$ is nilpotent if and only if $G$ has the normalizer property (also called normalizers grow). That is, every proper subgroup $H$ of $G$ is contained properly in its normalizer
$
N_G(H) = \{g \in G : H^g = H\}.
$

For infinite groups one implication still holds: every nilpotent group has the normalizer property. However, there are infinite groups with this property that are not even solvable. We show below that for groups definable in o-minimal structures nilpotency is equivalent to the
normalizer property for definable subgroups, provided $\mathcal{N}(G)$ is nilpotent:

\begin{prop}
Let $G$ be a definably connected group such that $\mathcal{N}(G)$
is nilpotent. Then $G$ is nilpotent if and only if $H \subsetneq N_G(H)$, for every proper definable $H < G$. 
\end{prop}

\begin{proof}
Assume $H \subsetneq N_G(H)$ for every proper definable $H < G$. We will show that $G$ is nilpotent. 

Recall that every definable group has a maximal normal definably connected solvable subgroup, called its solvable radical. If $G$ is not solvable, let $R$ be the solvable radical of $G$. Then the quotient of $G/R$ by its finite center is a centerless semisimple group $\bar{G}$. 

Suppose $\bar{G}$ is definably compact and let $H$ be the normalizer of a maximal definable torus $T$ of $\bar{G}$. We claim that $H$ is self-normalizing. Suppose $g \in \bar{G}$ normalizes $H$. Then $T^g$ is a maximal definable torus of $H$. Therefore $T^g = T^x$ for some $x \in H$, and $g \in H$ as well.   
Now the pull-back of $H$ in $G$ is a proper definable subgroup equal to its normalizer, contradiction.

If  $\bar{G}$ is not definably compact, then by \cite{me2}, $\bar{G} = \bar{K} \bar{H}$, where $\bar{K}$ is definably compact and  $\bar{H}$ is torsion-free. By \cite{PPSIII}, $G$ is elementarily equivalent to a connected centerless semisimple Lie group, for which maximal compact subgroups are self-normalizing subgroups. Therefore the pre-image of $\bar{K}$ in $G$ is be a proper definable subgroup equal to its normalizer, contradiction.

Hence $G$ must be solvable. If $G$ is not torsion-free, let $A$ be a $0$-Sylow of $G$. Then $G = \mathcal{N}(G) \cdot A$ by Fact \ref{me-solvable}.  Let $H = N_G(A)$. If $H = G$, then $A$ is normal in $G$. By Theorem \ref{finite-like}, then $G$ is nilpotent, and we are done. Assume that $H$ is a proper subgroup of $G$ (which by Theorem \ref{finite-like}, is equivalent to say that $G$ is not nilpotent). We claim that 
$N_{G} (H) = H$. Since $A$ is normal in $H$, then by Theorem \ref{finite-like}(2), $H$ is nilpotent. 
Let now $g \in G$ be such that $H^g = H$. As  $H$ is nilpotent,  by Proposition \ref{decomposition}, $A$ is the only $0$-Sylow of $G$ and $A^g = A$. Therefore $g \in N_{G} (A) = H$. So $H$ is a proper definable subgroup of $G$ equal to its normalizer, contradiction.
%

Thus we have shown that every time $G$ is not nilpotent, there is a definable subgroup $H < G$ such that $N_{G} (H) = H$.
\end{proof} 

As mentioned in the introduction, we do not know whether the assumption $\mathcal{N}(G)$ nilpotent in the Proposition above can be removed. 
 That is:

\begin{que} 
Is there a torsion-free definable group with the normalizer property for definable subgroup that is not nilpotent?
\end{que}

Proposition 3.1 finishes the proof of Theorem \ref{finite-like}.

\vs
\section{Nilpotent groups and linear algebraic groups}

\noindent
Connected solvable Lie groups that are definable in an o-minimal expansion of the reals are completely characterized in \cite{COS}. Some of them, for instance the group in \cite{Tao} pg. 327, are not Lie isomorphic to any linear algebraic group. However, if we restrict to nilpotent groups, the only definable Lie groups are linear algebraic:


\begin{proof}[Proof of Theorem \ref{Lie}] Clearly linear algebraic groups over the reals are definable
in the real field. Conversely, let $G$ be a nilpotent real Lie group definable in an o-minimal structure. 

First assume $G$ is connected.  
By Proposition \ref{decomposition}, $G$ has a definable torsion-free subgroup
$N= \mathcal{N}(G)$ and a central connected compact subgroup $T$ such that $G = N \times T$. Note that since $N$ is definable torsion-free, then it is closed and simply-connected (Fact \ref{tfref}).
By Theorem 4.5 in \cite{COS}, $N$ is a triangular group, so is isomorphic to a closed connected subgroup
of $UT_m(\R)$, the group of unipotent upper triangular matrices, for some $m \in \N$. All such groups are algebraic, as the exponential map is polynomial for nilpotent Lie algebras. If $\dim T = k$, then the subgroup $T$ is 
Lie isomorphic to the algebraic group $SO_2(\R)^k$.

If $G$ is not definably connected, then by \cite{Edmundo}, $G = F \cdot G^0$, for some finite normal subgroup $F$. By the connected case, $G^0$ is linear algebraic. Since finite groups are linear algebraic, so is $G$.
Therefore a definable real nilpotent  $G$ is Lie isomorphic to $\mathbb{U} \times \mathbb{K}$, where $\mathbb{U}$ is a closed connected subgroup of some $UT_m(\R)$, and $\mathbb{K}$, the maximal compact subgroup of $G$, is isomorphic to $F \cdot SO_2(\R)^k$, for some finite nilpotent $F$.
\end{proof}

By Theorem \ref{first} and results of Hrushovski, Peterzil and Pillay \cite{NIPI, HPP} on compact groups, the problem of determining whether a definable nilpotent group is elementarily equivalent to a linear algebraic group reduces to the torsion-free case.

By \cite{PPSIII}, every linearizable abelian torsion-free definable group can be decomposed into the product of definable 1-dimensional subgroups. This definable splitting has been proved also in \cite{PeSta} for groups definable in several o-minimal structures, and by an induction argument it reduces to the 2-dimensional case:

\begin{conj} \label{2dim-conj}
Every abelian 2-dimensional torsion-free group definable in an o-minimal structure \M\ is the product of two definable 1-dimensional subgroups. 
\end{conj}

 It is unknown whether Conjecture \ref{2dim-conj} holds in an arbitrary o-minimal structure. 
 We give below a positive answer for groups with an infinite definable family of definable automorphisms:
 
\begin{prop} 
 Let $(G, +)$ be an abelian 2-dimensional torsion-free group definable in an o-minimal structure \M, and let $\Aut(G)$ be the group of \M-definable automorphisms of $G$. If $\dim \Aut(G) > 0$, then $G$ can be decomposed as a direct product of definable 1-dimensional subgroups.
\end{prop}

\begin{proof}
We know by \cite{PS} that $G$ has a 1-dimensional definable subgroup $H$. 
Suppose $A$ is a different 1-dimensional definable subgroup of $G$. Then $A$ is a definable complement of $H$, and we are done. This is because $A \cap H = \{0\}$, as both $A$ and $H$ have no proper non-trivial definable subgroups,  and $H + A = G$, because $H+A$ is a definable subgroup of full dimension, and $G$ is definably connected.

So assume for a contradition that $H$ is the only non-trivial proper definable subgroup of $G$, and set $\bar{G} = G/H$. Thus $H$ is definably characteristic 
and
%
%
for each $x \in G, x \notin H$, $G = \langle x \rangle$ 
and each definable homomorphism from $G$ is determined by its value on $x$. 

 
\begin{lem} \label{embedding-quotient}
Let $\varphi_1, \varphi_2 \in \Aut(G)$, and let $ \bar{\varphi}_1, \bar{\varphi}_2 \in \Aut(\bar{G})$ be the induced maps on the quotient $\bar{G}$. Then
\[
\bar{\varphi}_1 = \bar{\varphi}_2 \quad \Longrightarrow\ \quad \varphi_1 = \varphi_2
\]

\vs \noindent
Therefore $\Aut(G) \hookrightarrow \Aut(\bar{G})$. 
\end{lem}
 
\begin{proof}
Let $x \in G \backslash H$, so that $G = \langle x \rangle$. Then 
\[
\varphi_1(x) = \varphi_2(x) + h, \quad \mbox{for some } h \in H
\]
because $\bar{\varphi}_1 = \bar{\varphi}_2$. Consider now the kernel of the homomorphism 
$\varphi_1 - \varphi_2$:
\[
K = \ker(\varphi_1 - \varphi_2) = \{g \in G : \varphi_1(g) = \varphi_2(g)\}.
\]
If $K = \{0\}$ then $\varphi_1 - \varphi_2 \in \Aut(G)$ and $(\varphi_1 - \varphi_2)(x) = h \in H$, impossible. Then $K$ is a non-trivial definable subgroup of $G$.  Since we are assuming that $H$ is the only non-trivial proper definable subgroup of $G$, it follows that $H \subseteq K$. As $\bar{\varphi}_1 = \bar{\varphi}_2$, then $\varphi_1 = \varphi_2$.
\end{proof}

\noindent
As $\dim \Aut(G) > 0$, there is an infinite definable family $F$ in $\Aut(G)$. Let $\bar{F} \subset \Aut(\bar{G})$ be the induced definable family on  the quotient $\bar{G}$. By Lemma \ref{embedding-quotient}, we know that $\bar{F}$ is infinite as well. By \cite{Peterzil-Starchenko1}, there is a definable product $\cdot$ on $\bar{G}$, such that $(\bar{G}, +, \cdot)$ is a definable field. We show below that $\Aut(G)$ is a 1-dimensional definable group, and it is definably isomorphic to the multiplicative group of $\bar{G}$, $\bar{G}^{\ast} = \bar{G} \backslash \{0\}$:

\begin{lem}
$\Aut(G) \cong (\bar{G}^{\ast}, \cdot)$.
\end{lem}

\begin{proof}
First let us see that $\Aut(\bar{G})$ is a definable group   definably isomorphic to $(\bar{G}^{\ast}, \cdot)$.
Let $f \in \Aut(\bar{G})$ and let $f(1) = a \in \bar{G}^{\ast}$. The set $\{x \in \bar{G} : f(x) = a \cdot x\}$ is a definable subgroup of $(\bar{G}, +)$ containing $0$ and $1$; but $(\bar{G}, +)$ does not have any proper definable subgroups, so $f(x) = a \cdot x$ for every $x \in \bar{G}$. On the other hand, every definable function $\bar{G} \to \bar{G}$ of the form $f(x) = a \cdot x$, with $a \in \bar{G}^{\ast}$, is a definable automorphism of $(\bar{G}, +)$, so $\Aut(\bar{G}) \cong (\bar{G}^{\ast}, \cdot)$. 

By Lemma \ref{embedding-quotient}, $\dim \Aut(G) = 1$ as well, and $\Aut(G)^0 \cong \bar{G}^{>0}$. Moreover
(for instance) $-id_G \mapsto -1 \in (\bar{G}^{\ast}, \cdot)$, so $\Aut(G) \cong (\bar{G}^{\ast}, \cdot)$.
\end{proof}

Fix now $x \in G$, $x \notin H$, and consider the set
\[
X = \{\varphi(x) : \varphi \in \Aut(G)\}.
\]

\vs \noindent
Clearly $X$ is a definable set, and $\dim X = \dim \Aut(G) = 1$, because $x$ is a generator of $G$.
Moreover $X \cap H = \emptyset$, because no element in $H$ is a generator. We claim that $K = X \cup \{0\}$ is a subgroup: \\

\begin{itemize}

\item $a \in K \Rightarrow -a \in K$, because if $\varphi \in \Aut(G)$, then $-\varphi \in \Aut(G)$. \\

\item $a, b \in K \Rightarrow a+b \in K$: \\

\begin{enumerate}
\item[(i)] If $b = -a$, then $a + b = 0$, and there is nothing to prove. \\

\item[(ii)] Let $b \neq -a$, with $\varphi(x) = a$ and $\psi(x) = b$. We claim that $\varphi + \psi \in \Aut(G)$. Otherwise
\[
F = \ker(\varphi + \psi )  = \{g \in G : \varphi(g) = - \psi(g)\}
\]
would be a proper (because $\varphi(x) \neq -\psi(x)$) non-trivial definable subgroup of $G$, so $H = F$. Therefore $f = (-\psi)^{-1} \circ \varphi$
would be a definable automorphism of $G$ that is the identity on $H$, and is not the identity on $G$. So
consider the set of all such automorphisms of $G$:
\[
Y = \{\varphi \in \Aut(G) : \varphi_{|_H} = \id_H \}
\]

Now $Y$ would be an infinite (because it contains $f$ and all its powers) definable subgroup of $\Aut(G)$. By dimension reasons $Y^0 = \Aut(G)^0$, which is impossible, because $\Aut(G)^0$ contains all multiplications by positive rational numbers, none of which is the identity on $H$. \\

Therefore $\varphi + \psi \in \Aut(G)$, and $(\varphi + \psi)(x) = a+b$.
\end{enumerate}
\end{itemize}

\vs \noindent
So we have shown that if $\dim \Aut(G) > 0$, then the definable 1-dimensional subroup $H$ has a definable complement in $G$, as wanted.
\end{proof}

\begin{que}
What if $\dim \Aut(G) = 0$?
\end{que}


We conclude with a remark and a question about the general case:

\begin{rem}
Let $G$ be a definably connected group in an o-minimal structure. Assume $G$ is elementarily equivalent to a real algebraic group. Then $\mathcal{N}(G)$ is nilpotent and $G$ has a definable Levi decomposition.
\end{rem}

\begin{proof}
In real algebraic groups any normal closed connected simply-connected subgroup is nilpotent, so $\mathcal{N}(G)$ must be nilpotent. Moreover, in real algebraic groups the intersection between the solvable radical and any Levi subgroup is finite, therefore Levi subgroups of $G$ from \cite{Conversano-PillayIII} must be definable.
\end{proof}

\begin{que}
Let $G$ be a definably connected group in an o-minimal structure such that $\mathcal{N}(G)$ is nilpotent and $G$ has a definable Levi decomposition. Is $G$ elementarily equivalent to a real algebraic group?
\end{que}
 
 {\bf Acknowledgments.} Thanks to Anand Pillay for suggesting (some time ago) to study nilpotent groups, and to Igor Klep for reading and commenting on an earlier version of this paper. Thanks also to the anonymous referee for their careful reading and insightful suggestions/questions.


\end{document}